\documentclass[12pt,a4paper]{amsart}

\usepackage[utf8]{inputenc}
\usepackage[T1]{fontenc}
\usepackage[english]{babel}
\usepackage{cite}

\usepackage{indentfirst}
\usepackage{amssymb}
\usepackage{amsfonts}
\usepackage{amsmath}
\usepackage{amsthm}
\usepackage[top=2.5cm, bottom=2.5cm, left=2.5cm, right=2.5cm]{geometry}
\usepackage{amsopn}
\usepackage[colorlinks]{hyperref}
\usepackage{mathrsfs}
\usepackage{amsxtra}
\usepackage{color}
\usepackage{dsfont}

\newcommand{\al}{\alpha}
\newcommand{\be}{\beta}
\newcommand{\la}{\lambda}
\newcommand{\Vd}{\mathcal{V}_{d}}
\newcommand{\ga}{\gamma}
\newcommand{\va}{\varepsilon}

\theoremstyle{plain}
\newtheorem{thm}{Theorem}

\newtheorem{prop}[thm]{Proposition}

\theoremstyle{definition}
\newtheorem{defn}[thm]{Definition}

\newtheorem*{example*}{Example}

\newtheorem*{rem*}{Remark}
\newtheorem{rem}[thm]{Remark}

\newcommand{\divv}{\text{div}\,}

\newcommand{\R}{\mathbb{R}}

\DeclareSymbolFont{bbsymbol}{U}{bbold}{m}{n}
\DeclareMathSymbol{\ind}{\mathbin}{bbsymbol}{'061}
\title[Improved Hardy  and Rellich inequalities]{Improved Hardy and Rellich inequalities for antisymmetric and odd functions}
\author[M{.} Kijaczko]{Micha\l{} Kijaczko}


\address[ M.K.]{Faculty of Pure and Applied Mathematics\\ Wroc{\l}aw University 
	of Science and Technology\\
	Wybrze\.ze Wyspia\'nskiego 27,
	50-370 Wroc{\l}aw, Poland
}
\email{michal.kijaczko@pwr.edu.pl}

\begin{document}

\begin{abstract}
We investigate the Hardy and Rellich inequalities for classes of antisymmetric and odd functions and general exponent $p$. The obtained constants are better than the classical ones.
\end{abstract}

	\maketitle
	
\section{Introduction}
\subsection{Hardy inequalities}
The classical Hardy inequality on $L^p(\R^d)$ states that 
\begin{equation}\label{classicalHardy}
    \int_{\R^d}|\nabla u(x)|^p\,dx\geq\left(\frac{d-p}{p}\right)^p\int_{\R^d}\frac{|u(x)|^p}{|x|^p}\,dx,\text{  }u\in W^{1,p}(\R^d),\text{  }1\leq p<d.
\end{equation}
If $p>d$, one has
\begin{equation}\label{classicalHardyp>d}
    \int_{\R^d}|\nabla u(x)|^p\,dx\geq\left(\frac{p-d}{p}\right)^p\int_{\R^d}\frac{|u(x)|^p}{|x|^p}\,dx,\text{  }u\in W^{1,p}(\R^d\setminus\{0\}).
\end{equation}
Hardy inequalities play a crucial role in analysis and partial differential equations. The literature concerning them is enormous; we refer the Reader for example to \cite{MR1747888, MR3966452,MR1769903} and references therein.

Recently, there has been a research interest \cite{MR4422783,hardysobanty,gupta} connected to the development of the Hardy inequality for antisymmetric functions. The antisymmetry condition reads
\begin{equation}\label{antisymmetry}
u(x_1,\dots,x_i,\dots,x_j,\dots,x_d)=-u(x_1,\dots,x_j,\dots,x_i,\dots,x_d),\,\,1\leq i\neq j\leq d.
\end{equation}
 Hoffmann-Ostenhof and Laptev \cite[Theorem 1]{MR4422783} established the following improved Hardy inequality for antisymmetric functions,
\begin{equation}\label{hardyantisymmetricp=2}
\int_{\R^d}|\nabla u(x)|^2\,dx\geq\left(\frac{d^2-2}{2}\right)^2\int_{\R^d}\frac{|u(x)|^2}{|x|^2}\,dx.
\end{equation}
The constant $(d^2-2)^2/4$ in \eqref{hardyantisymmetricp=2} is optimal, i.e. it cannot be replaced by any bigger number. Recently, Gupta \cite{gupta} gave an alternative proof of \eqref{hardyantisymmetricp=2} and also proved the discrete version of the antisymmetric Hardy inequality on the lattice $\mathbb{Z}^d$ with asymptotically sharp constant.

Another class of improved Hardy inequalities can be obtained for odd functions, i.e. satisfying $u(-x)=-u(x)$. In \cite{MR2389919} and \cite{MR3335687} the following Hardy inequality for odd functions was derived,
\begin{equation}\label{Hardyoddp=2}
 \int_{\R^d}|\nabla u(x)|^2\,dx\geq\frac{d^2}{4}\int_{\R^d}\frac{|u(x)|^2}{|x|^2}\,dx. 
\end{equation}
The constant $d^2/4$ in \eqref{Hardyoddp=2} is also sharp.

However, the mentioned works only deal with the case $p=2$. In this article we extend the above results to general exponents $p\geq 2$, that is we prove the Hardy inequalities for antisymmetric and odd functions on $L^p(\R^d)$ with improved constants. To this end we need first to define the antisymmetric and odd Sobolev spaces.
\begin{defn}
Let $p\geq 1$, $k=1,2,\dots$. We define the antisymmetric and odd Sobolev spaces $W^{k,p}_A(\R^d)$ and $W^{k,p}_{O}(\R^d)$, respectively, as
\begin{align*}
W^{k,p}_{A}(\R^d)&=\{u\in W^{k,p}(\R^d):\,u\text{ is antisymmetric}\},\\
W^{k,p}_{O}(\R^d)&=\{u\in W^{k,p}(\R^d):\,u\text{ is odd}\}.
\end{align*}
\end{defn}
We clearly have $W^{k,p}_A(\R^d)\subset W^{k,p}(\R^d)$ and $W^{k,p}_{O}(\R^d)\subset W^{k,p}(\R^d)$, hence, one should expect that the Hardy constant can be enlarged for these classes of functions. This is indeed true and is made precise in the next two theorems.
\begin{thm}[\textbf{Hardy inequality for antisymmetric functions}]\label{tw1}
Let $p\geq 2$, $d\geq 2$. Let $u\in W^{1,p}_{A}(\R^d)$. Then, the following improved Hardy inequality holds true,
\begin{equation}\label{Hardyantisymmetric}
\int_{\R^d}|\nabla u(x)|^p\,dx\geq C_H(d,p)\int_{\R^d}\frac{|u(x)|^p}{|x|^p}\,dx.    
\end{equation}
The constant $C_H(d,p)$ is given by
$$
C_H(d,p)=\left(\frac{2(p-2)d(d-1)}{p^2}+\left(\frac{d^2-p}{p}\right)^2\right)^{\frac{p}{2}}.
$$
\end{thm}

\begin{thm}[\textbf{Hardy inequality for odd functions}]\label{tw2}
Let $p\geq 2$, $d\geq 1$. Let $u\in W^{1,p}_{O}(\R^d)$. Then, the following improved Hardy inequality holds true,
\begin{equation}\label{Hardyodd}
\int_{\R^d}|\nabla u(x)|^p\,dx\geq D_H(d,p)\int_{\R^d}\frac{|u(x)|^p}{|x|^p}\,dx.    
\end{equation}
The constant $D_H(d,p)$ is given by
$$
D_H(d,p)=\left(\frac{4(p-2)}{p^2}+\left(\frac{d-p+2}{p}\right)^2\right)^{\frac{p}{2}}.
$$
\end{thm}

We do not know if our constants $C_H(d,p)$ and $D_H(d,p)$ are optimal or what is happening in the remaining case $1<p<2$. Firstly, our methods does not work for $1<p<2$. Moreover, it seems to be difficult to prove the sharpness of the constants. The proof of the sharp antisymmetric Hardy inequality for $p=2$ from \cite{MR4422783} relies on orthogonal expansions into spherical harmonics. Such methods are usually not applicable in the $L^p$ setting. In addition, it is hard to obtain the explicit expressions for $L^p$ norms of non-radial functions. Nevertheless, the constant $C_H(d,p)$ in \eqref{Hardyantisymmetric} has the asymptotic $\left(d^2/p\right)^p$, while the constant from classical Hardy inequality is of order $\left(d/p\right)^p$, as $d\rightarrow\infty$. The constant $D_H(d,p)$ from \eqref{Hardyodd} has the same growth as the classical Hardy constant. Observe that both constants coincide for $d=2$. One can quickly check that both $C_H(d,p)$ and $D_H(d,p)$ are strictly greater than $\left(|d-p|/p\right)^p$ with one exception: for $d=1$ the constant $D_H(1,p)$ equals $\left((p-1)/p\right)^p$. On the other side, for fixed $d$ it holds
$$
\lim_{p\rightarrow\infty}\left(\frac{p-d}{p}\right)^p=\lim_{p\rightarrow\infty}C_H(d,p)=\lim_{p\rightarrow\infty}D_H(d,p)=e^{-d},
$$
which shows that in this setting all constants are asymptotically equivalent, when $p\rightarrow\infty$.

Moreover, the classical Hardy constant $\left(|d-p|/p\right)^p$ equals zero for $p=d$, that is in this case the Hardy inequality \eqref{classicalHardy} is trivial. However, the constants $C_H(p,p)$ and $D_H(p,p)$ do not vanish. This is a great advantage of considering the narrower classes of functions.

With a slight modification of the proof, one can easily obtain weighted versions of the Hardy inequalities \eqref{Hardyantisymmetric} and \eqref{Hardyodd}. This result is stated in the next theorem.
\begin{thm}[\textbf{Weighted Hardy inequalities for antisymmetric and odd functions}]\label{Hardyweighted}
Let $\gamma\in\R$ and let $p\geq 2$.\\
i) If $u\in C^1_c(\R^d\setminus\{0\})$ is antisymmetric and $d\geq 2$, then 
$$
\int_{\R^d}\frac{|\nabla u(x)|^p}{|x|^{\gamma}}\,dx\geq C_H(d,p,\gamma)\int_{\R^d}\frac{|u(x)|^p}{|x|^{p+\gamma}}\,dx,
$$
where 
$$
C_H(d,p,\gamma)=\left(\frac{2(p-2+\gamma)d(d-1)}{p^2}+\left(\frac{d^2-p-\gamma}{p}\right)^2\right)^{\frac{p}{2}};
$$
ii) if  $u\in C^1_c(\R^d\setminus\{0\})$ is odd and $d\geq 1$, then
$$
\int_{\R^d}\frac{|\nabla u(x)|^p}{|x|^{\gamma}}\,dx\geq D_H(d,p,\gamma)\int_{\R^d}\frac{|u(x)|^p}{|x|^{p+\gamma}}\,dx,
$$
where 
$$
D_H(d,p,\gamma)=\left(\frac{4(p-2+\gamma)}{p^2}+\left(\frac{d-p-\gamma+2}{p}\right)^2\right)^{\frac{p}{2}}.
$$
\end{thm}
\begin{rem}
For the most important case $p=2$  we have
\begin{align*}
C_H(d,2,\ga)&=\frac{\ga\,d(d-1)}{2}+\left(\frac{d^2-2-\ga}{2}\right)^2,\\
D_H(d,2,\ga)&=\ga+\left(\frac{d-\ga}{2}\right)^2.
\end{align*}
These constants are sharp: this fact can be proved the exact same way as in \cite[Proposition 2]{MR4422783} by considering functions being products of the radial part and an appropriate spherical harmonic.
\end{rem}
\subsection{Rellich inequalities}
Our next goal is to obtain improved versions of Rellich inequality for classes of antisymmetric and odd functions. The classical Rellich inequality \cite{MR0088624} states that for any dimension $d\geq 5$ and $u\in W^{2,2}(\R^d)$,
\begin{equation}\label{Rellichp=2}
\int_{\R^d}|\Delta u(x)|^2\,dx\geq\frac{d^2(d-4)^2}{16}\int_{\R^d}\frac{|u(x)|^2}{|x|^4}\,dx.
\end{equation}
Davies and Hinz \cite{MR1612685} extended \eqref{Rellichp=2} to the case of arbitrary $p>1$. The (most general) weighted $L^p$ version of the Rellich inequality can be found in Mitidieri's work \cite[Theorem 3.1]{MR1769903}: for any $p>1$, $d\geq 1$, $-(p-1)d<\gamma<d-2p$ and $u\in C_c^{2}(\R^d\setminus\{0\})$, the following inequality holds true,
\begin{align*}\label{weightedRellich}
\int_{\R^d}\frac{|\Delta u(x)|^p}{|x|^{\gamma}}\,dx\geq c(p,d,\gamma)\int_{\R^d}\frac{|u(x)|^p}{|x|^{\gamma+2p}}\,dx,   
\end{align*}
where the sharp constant $c(d,p,\gamma)$ equals
$$
c(p,d,\gamma)=\left(\frac{(d-\gamma-2p)((p-1)d+\gamma)}{p^2}\right)^p.
$$
For more information about Rellich inequlities we refer the Reader for example to \cite{MR1612685,MR3934125} and references therein. A research connected to our setting can also be found in \cite{MR4543806}.

In this article we prove Rellich inequalities with better constants for antisymmetric and odd functions. These results read as follows.
\begin{thm}[\textbf{Weighted Rellich inequality for antisymmetric functions}]\label{WeightedRellichanti}
Let $p>1$, $d\geq 2$ and $u\in C_c^2(\R^d\setminus\{0\})$ be antisymmetric. Then the following weighted Rellich inequality holds,
\begin{equation}\label{weightedRellichantisymmetric}
\int_{\R^d}\frac{|\Delta u(x)|^p}{|x|^{\gamma}}\,dx\geq C_R(p,d,\ga)\int_{\R^d}\frac{|u(x)|^p}{|x|^{\ga+2p}}\,dx,
\end{equation}
where 
$$
C_R(d,p,\ga)=\left(\frac{(\ga+2p-2)\left(2(p-1)d(d-1)+p(d-\ga-2p)\right)+(p-1)\left(d^2-\ga-2p\right)^2}{p^2}\right)^p,
$$
provided that $\ga$ is such that the numerator of the above is nonnegative.
\end{thm}

\begin{thm}[\textbf{Weighted Rellich inequality for odd functions}]\label{WeightedRellichodd}
Let $p>1$, $d\geq 1$ and $u\in C_c^2(\R^d\setminus\{0\})$ be odd. Then the following weighted Rellich inequality holds,
\begin{equation}\label{weightedRellichodd}
\int_{\R^d}\frac{|\Delta u(x)|^p}{|x|^{\gamma}}\,dx\geq D_R(p,d,\ga)\int_{\R^d}\frac{|u(x)|^p}{|x|^{\ga+2p}}\,dx,
\end{equation}
where 
$$
D_R(d,p,\ga)=\left(\frac{(\ga+2p-2)(4(p-1)+p(d-\ga-2p))+(p-1)(d-\ga-2p+2)^2}{p^2}\right)^p.
$$
provided that $\ga$ is such that the numerator of the above is nonnegative.
\end{thm}
The conditions on $\ga$ in the inequalities above can be stated explicitly, but take a complicated form. Similarly as in the case of Hardy inequalities, we do not know if these constants are sharp, unless $p=2$. By iteration, one can quickly obtain appropriate antisymmetric and odd Rellich inequalities for the higher powers of the Laplacian, which are of the form
$$
\int_{\R^d}\frac{|\Delta^m u(x)|^p}{|x|^{\ga}}\,dx\geq C(d,\ga,p,m)\int_{\R^d}\frac{|u(x)|^p}{|x|^{\ga+2mp}}\,dx,
$$
with explicit constants.

For $p=2$ the constants $C_R(d,p)$ and $D_R(d,p)$ simplify to
\begin{align*}
C_R(d,2,\ga)&=\frac{\left(d^4-4d^2-\ga(\ga+4)\right)^2}{16},\,-d^2<\ga<d^2-4,\\
D_R(d,2,\ga)&=\frac{\left(d^2-(\ga+2)^2\right)^2}{16},\,-d-2<\ga<d-2.
\end{align*}
 In particular, for the most important case $\ga=0$, we obtain two Rellich inequalities, which deserve to be formulated explicitly:
\begin{align*}
\int_{\R^d}|\Delta u(x)|^2\,dx&\geq\frac{d^4(d^2-4)^2}{16}\int_{\R^d}\frac{|u(x)|^2}{|x|^4}\,dx,\,u\in W^{2,2}_{A}(\R^d),\,d\geq 3,\\
\int_{\R^d}|\Delta u(x)|^2\,dx&\geq\frac{\left(d^2-4\right)^2}{16}\int_{\R^d}\frac{|u(x)|^2}{|x|^4}\,dx,\,u\in W^{2,2}_{O}(\R^d),\,d\geq 3.    
\end{align*}
\begin{rem}
Compared to the classical Rellich inequality \eqref{Rellichp=2}, we may allow to have the dimensions $d=3,4$ in the inequalities above. Moreover, the constant from the antisymmetric Rellich inequality grows like $d^8/16$, as $d\rightarrow\infty$, in contrast to the growth of the order $d^4/16$ from \eqref{Rellichp=2}. In general, we have 
$$C_R(d,p,\ga)\sim\left(\frac{(p-1)d^4}{p^2}\right)^p$$ 
and  
$$D_R(d,p,\ga)\sim\left(\frac{(p-1)d^2}{p^2}\right)^p.$$
\end{rem}
\subsection{Applications}
Our improved Hardy inequalities can be applied to obtain results concerning spectral properties of the Schrödinger-type operator with Dirichlet $p$-Laplacian. Recall that, for $p>1$, the latter is defined as 
$$
-\Delta_p u=-\divv\left(\left|\nabla u\right|^{p-2}\nabla u\right),\,u\in C_c^{\infty}(\R^d).
$$
For a given nonnegative measurable potential $V$ we define the operator
$$
Hu=-\Delta_pu-Vu|u|^{p-2},\,u\in C_c^{\infty}(\R^d)
$$
and its quadratic form
$$
\langle Hu,u\rangle=\int_{\R^d}\left(|\nabla u(x)|^p-V(x)|u(x)|^p\right)\,dx.
$$
When $V(x)=|x|^{-p}$ is the Hardy potential, our Theorems \ref{tw1} and \ref{tw2} state that 
\begin{align*}
-\Delta_p\geq C_H(d,p)\frac{|\cdot|^{p-2}}{|x|^p},\,u\in W_{A}^{1,p}(\R^d),\\
-\Delta_p\geq D_H(d,p)\frac{|\cdot|^{p-2}}{|x|^p},\,u\in W_{O}^{1,p}(\R^d)
\end{align*}
in the sense of quadratic form, that is, the operators $H_1=-\Delta_p-C_H(d,p)V|\cdot|^{p-2}$ and $H_2=-\Delta_p-D_H(d,p)V|\cdot|^{p-2}$ are nonnegative on $W_{A}^{1,p}(\R^d)$ and $W_{O}^{1,p}(\R^d)$, respectively. For general potentials $V$, one can try to find conditions under which the operator $H$ is nonnegative, using the same approach as in \cite{MR4422783}.

Similar properties can be derived from our improved Rellich inequalities as well. The $p$-Laplace operator must be then replaced by $-\Delta\left(|\Delta u|^{p-2}\Delta u\right)$. In particular, when $p=2$, the above can be applied to study fourth-order elliptic and parabolic PDE's, connected to the Schrödinger-type operator $(-\Delta)^2-V$; see e.g. \cite{OwenMark} for details.\\

%The paper is organized as follows. In Section \ref{Section2} we provide basic facts concerning the density of smooth, compactly supported functions in antisymmetric and odd Sobolev spaces and the Vandermonde function. Next, we give a proof of the antisymmetric Hardy inequality \eqref{Hardyantisymmetric}. Section \ref{Section3} contains the proof of the odd Hardy inequality \eqref{Hardyodd}. In Section \ref{Section4} we focus on Rellich inequalities, Section \ref{Section5} is devoted to the applications of obtained inequalities. Finally, in the Appendix we briefly discuss the possible form of the optimal constants in antisymmetric and odd Hardy inequalities.
\section{Hardy inequality for antisymmetric functions}\label{Section2}
\subsection{Approximation by smooth, compactly supported functions}
We begin this section by stating a useful density fact, which we will need later on. It is well known that $C_c^{\infty}(\R^d)$ is always dense in $W^{k,p}(\R^d)$, for any $p\geq 1$ and $d,\,k=1,2,\dots$. Here we briefly discuss an analogous property for the subspaces $W^{k,p}_{A}(\R^d)$ and $W^{k,p}_{O}(\R^d)$. Let $\{\eta_{\va}\}_{\va>0}$ be a~ standard family of radial mollifiers and let $\varphi$ be a radial bump function of a class $C_c^{\infty}(\R^d)$. For a given $f\in W^{k,p}(\R^d)$ and $R>0$ we define
$$
f^{R,\va}(x)=\left(f*\eta_{\va}(x)\right)\varphi\left(\frac{x}{R}\right),
$$
where $h*g(x)=\int_{\R^d}h(y)g(x-y)\,dy$ is the standard convolution operation. We have $f^{R,\va}\in C_c^{\infty}(\R^d)$ and $f^{R,\va}\rightarrow f$ in $W^{k,p}(\R^d)$, as $R\rightarrow\infty$, $\va\rightarrow 0$. Moreover, from its definition it is clearly visible that $f^{R,\va}$ is antisymmetric, when $f$ is antisymmetric and $f^{R,\va}$ is odd, when $f$ is odd. Therefore, we obtain the following density property.
\begin{prop}\label{density}
Let $p\geq 1$, $k=1,2,\dots$.\\
\\
i)\quad If $d\geq 2$ then $C_c^{\infty}(\R^d)$ antisymmetric functions are dense in $W^{k,p}_{A}(\R^d)$;\\
\\
ii)\quad if $d\geq 1$, then $C_c^{\infty}(\R^d)$ odd functions are dense in $W^{k,p}_{O}(\R^d)$.
\end{prop}
\subsection{Vandermonde determinant and its properties}
Recall that, for $d=2,3,\dots$, the Vandermonde determinant is defined as
\begin{align*}
\Vd(x)=\Vd(x_1,\dots,x_d)=\begin{vmatrix}
1 & x_{1} & x_{1}^{2} & \dots & x_{1}^{d-1} \\ 
1 & x_{2} & x_{2}^{2} & \dots & x_{2}^{d-1} \\
% \hdotsfor{5} \\
\vdots & \vdots & \vdots & \ddots & \vdots \\
1 & x_{d} & x_{d}^{2} & \dots & x_{d}^{d-1}
\end{vmatrix}
=
\prod_{i < j} (x_{j} - x_{i}).
\end{align*}
This function is antisymmetric and homogeneous of order $d(d-1)/2$, that is $\Vd(ax)=a^{d(d-1)/2}\Vd(x)$ for all $x\in\R^d$, $a\in\R$. Moreover, $\Vd$ is harmonic, i.e. $\Delta\Vd=0$.

The following statement is elementary and essentially well known, but crucial to our further computations. For the sake of completeness, we provide its short proof.
\begin{prop}\label{prop1}
For $d=2,3,\dots$ and all $x\in\R^d$ we have 
\begin{equation}\label{Vandermonde_identity}
\sum_{i=1}^{d}x_i\frac{\partial\Vd}{\partial x_i}(x)=\frac{d(d-1)}{2}\Vd(x).
\end{equation}
\end{prop}
\begin{proof}
%Let $S_d$ be the group of permutations on $\{1,2,\dots,d\}$. Using the definition of the determinant, we have
%\begin{align*}
%\Vd(x)=\sum_{\sigma\in S_d}\sgn(\sigma)\prod_{j=1}^{d}x^{\sigma(j)-1}_{j}  
%\end{align*}
%and
%\begin{align*}
%\frac{\partial\Vd}{\partial x_i}(x)&=\sum_{\sigma\in S_d}\sgn(\sigma)\left(\sigma(i)-1\right) x_{i}^{\sigma(i)-2}\prod_{j\neq i}x^{\sigma(j)-1}_{j}.
%\end{align*}
%Therefore, 
%\begin{align*}
%\sum_{i=1}^{d}x_i\frac{\partial\Vd}{\partial x_i}(x)=\sum_{\sigma\in S_d}\sgn(\sigma)\prod_{j=1}^{d}x^{\sigma(j)-1}_{j}\sum_{i=1}^{d}\left(\sigma(i)-1\right)=\frac{d(d-1)}{2}\Vd(x),    
%\end{align*}
%because $\displaystyle\sum_{i=1}^{d}\left(\sigma(i)-1\right)=\sum_{i=1}^{d}\left(i-1\right)=\frac{d(d-1)}{2}$ for any $\sigma\in S_d$.

We differentiate both sides of the identity $\Vd(ax)=a^{d(d-1)/2}\Vd(x)$ with respect to $a$ and then put $a=1$.
\end{proof}
\begin{rem}
Notice that \eqref{Vandermonde_identity} and the Schwarz inequality imply that 
\begin{equation}\label{schwarz}
|x||\nabla\Vd(x)|\geq \frac{d(d-1)}{2}|\Vd(x)|,\,x\in\R^d.
\end{equation}
\end{rem}
The Vandermonde determinant appears naturally in considerations concering antisymmetric functions. As stated in \cite[Proposition 1]{MR4422783}, if $\varphi$ is an analytic, antisymmetric function on $\R^d$, then $\varphi=\psi\Vd$, where $\psi$ is analytic and symmetric.

\subsection{Proof of the antisymmetric Hardy inequality}
\begin{proof}[Proof of Theorem \ref{tw1}]
Let $u\in W^{1,p}_{A}(\R^d)$. Using Proposition \ref{density} we may and do assume that $u$ is smooth and compactly supported. Define the domain $\Omega$ by 
\begin{equation}\label{Omega}
\Omega=\{x\in\R^d:x_1<x_2\dots<x_d\}.
\end{equation}
The boundary of $\Omega$ is of course 
$$
\partial\Omega=\{x\in\R^d: x_i=x_j\text{ for some }i,j\in\{1,2,\dots,d\},i\neq j\},
$$
so it consists of $\tfrac{d(d-1)}{2}$ hyperplanes. Moreover, any antisymmetric function vanishes on $\partial\Omega$. Since the integrands are symmetric, we clearly have $$\int_{\R^d}|\nabla u(x)|^p\,dx=d!\int_{\Omega}|\nabla u(x)|^p\,dx$$ and $$\int_{\R^d}\frac{|u(x)|^p}{|x|^{p}}\,dx=d!\int_{\Omega}\frac{|u(x)|^p}{|x|^{p}}\,\,dx,$$ therefore it suffices to prove the corresponding Hardy inequality on the domain $\Omega$.  In order to complete the proof, we want to adapt the method from \cite{MR1769903} and \cite{MR2048514}. More precisely, if $T\colon\Omega\to\R^d,\,T=(T^1,\dots,T^d)$, is a vector field of a class $C^1$ such that all $T^k(x)|u(x)|^p$ vanish, when $x$ reaches the boundary of $\Omega$ (and infinity), then the following Hardy inequality holds,
\begin{equation}\label{generalHardyT}
\int_{\Omega}|\nabla u(x)|^p\,dx\geq\int_{\Omega}|u(x)|^p\left(\divv T(x)-(p-1)|T(x)|^\frac{p}{p-1}\right)\,dx.
\end{equation}
Let $\al\in\R$, $\be\geq 0$ be free parameters, to be chosen later and let $\varepsilon>0$. We define the vector fields $T_{\varepsilon}$ and $T$ on $\Omega$ by
$$
T_{\varepsilon}(x)=\frac{\al x}{|x|^p+\varepsilon}-\frac{\be\,\nabla\Vd(x)}{\left(\Vd(x)+\va\right)\left(|x|^{p-2}+\varepsilon\right)}
$$
and
$$
T(x)=\lim_{\varepsilon\rightarrow 0}T_{\varepsilon}(x)=\frac{\al x}{|x|^p}-\frac{\be\,\nabla\Vd(x)}{\Vd(x)|x|^{p-2}}.
$$
Observe first that $\Vd$ is positive on $\Omega$, hence, $T_{\va}$ and $T$ are well defined and, since $u$ is compactly supported in $\R^d$ and vanishes on $\partial\Omega$, then \eqref{generalHardyT} holds with $T_{\va}$ in place of $T$. Therefore, Fatou's lemma yields \eqref{generalHardyT} with $T$. 

Let $\lambda=\tfrac{d(d-1)}{2}$. A quick computation using Proposition \ref{prop1} and the harmonicity of $\Vd$ shows that
$$
\divv T(x)=\frac{\al(d-p)+\be (p-2)\la}{|x|^p}+\frac{\be\,|\nabla\Vd(x)|^2}{\Vd(x)^2 |x|^{p-2}}.
$$
Observe that if $p\geq 2$, then $\tfrac{p}{p-1}\leq 2$. Using \eqref{Vandermonde_identity}, we have
\begin{align*}
 |T(x)|^{\frac{p}{p-1}}&=\left(\frac{\al^2}{|x|^{2p-2}}-\frac{2\al\be\lambda}{|x|^{2p-2}}+\be^2\frac{|\nabla\Vd(x)|^{2}}{\Vd(x)^{2}|x|^{2p-4}}\right)^{\frac{p}{2(p-1)}}\\
 &=\frac{1}{|x|^p}\left[\al^2-2\al\be\la+\be^2\left(\frac{|x||\nabla\Vd(x)|}{\Vd(x)}\right)^2\right]^{\frac{p}{2(p-1)}}.
\end{align*}
Therefore, we may write that
\begin{align*}
\divv T(x)-(p-1)|T(x)|^\frac{p}{p-1}\geq\frac{f(t;\al,\be)}{|x|^p},
\end{align*}
where, by \eqref{schwarz},
$$
t=\left(\frac{|x||\nabla\Vd(x)|}{\Vd(x)}\right)^2\geq\lambda^2
$$
and
$$
f(t;\al,\be)=\al(d-p)+\be(p-2)\la+\be t-(p-1)\left(\al^2-2\al\be\la+\be^2 t\right)^{\frac{p}{2(p-1)}}.
$$
To derive the Hardy constant, we need to find the value of 
$$
\max_{\al,\be}\left(\min_{t\geq\lambda^2}f(t;\al,\be)\right).
$$
Let us define the auxiliary function
$$
g(t;\al,\be)=\be t-(p-1)\left(\al^2-2\al\be\la+\be^2 t\right)^{\frac{p}{2(p-1)}}.
$$
We easily compute the partial derivatives of $g$ with respect to $t$, that is
\begin{align*}
\frac{\partial g}{\partial t}=\be\left(1-\frac{\be p}{2}\left(\al^2-2\al\be\la+\be^2 t\right)^{\frac{2-p}{2(p-1)}}\right)   
\end{align*}
and, recalling that $p\geq 2$,
$$
\frac{\partial^2 g}{\partial t^2}=\frac{p(p-2)\be^4}{4(p-1)}\left(\al^2-2\al\be\la+\be^2 t\right)^{\frac{4-3p}{2(p-1)}}\geq 0.
$$
Hence, standard calculus shows that  the minimum of the function $t\mapsto g(t;\al,\be)$ is attained at 
$$
t_0=\frac{\left(\frac{p\be}{2}\right)^{\frac{2(p-1)}{p-2}}-\al^2+2\al\be\la}{\be^2}.
$$
We have $t_0\geq\la^2$ if and only if
\begin{equation}\label{albe}
|\al-\la\be|\leq \left(\frac{p\be}{2}\right)^{\frac{p-1}{p-2}}.   
\end{equation}
Therefore, assuming \eqref{albe}, we obtain that
$$
g(t;\al,\be)\geq g(t_0;\al,\be)=:G(\al,\be),
$$
with
$$
G(\al,\be)=\al(d-p)+\be(p-2)\la+2\al\la-\frac{\al^2}{\be}-\be^{\frac{p}{p-2}}\left(\frac{p}{2}\right)^{\frac{p}{p-2}}\left(\frac{p}{2}-1\right).
$$
The global maximum of $G$ is attained at the point $P=(\al_0,\be_0)$, where
\begin{align*}
\al_0&=\left(\frac{d-p+2\la}{2}\right)\left[(p-2)\la+\left(\frac{d-p+2\la}{2}\right)^2\right]^{\frac{p-2}{2}}\left(\frac{2}{p}\right)^{p-1},\\
\be_0&=\left[(p-2)\la+\left(\frac{d-p+2\la}{2}\right)^2\right]^{\frac{p-2}{2}}\left(\frac{2}{p}\right)^{p-1}.
\end{align*}
The condition \eqref{albe} is clearly fulfilled, because $|\al_0-\la\be_0|=|\tfrac{d-p}{2}|\be_0$. Moreover, the maximum equals 
$$
G(\al_0,\be_0)=\left(\frac{2}{p}\right)^p\left[(p-2)\la+\left(\frac{d-p+2\la}{2}\right)^2\right]^{\frac{p}{2}}.
$$
Recalling that $\la=\tfrac{d(d-1)}{2}$ ends the proof.
\end{proof}
\section{Hardy inequality for odd functions}\label{Section3}
We begin with a simple observation concerning odd functions.
\begin{prop}\label{propositionodd}
Let $u\colon\R^d\to\R$ be a $C^1$ odd function, not constantly equal to zero. Then, there exists an odd polynomial $h$ of order one and an even function $\varphi$, both not constantly equal to zero, such that \\
\\
i)\quad$u(x)=h(x)\varphi(x),\,\,x\neq 0$;\\
\\
ii)\quad$\displaystyle\sum_{k=1}^{d}x_k\frac{\partial h}{\partial x_{k}}(x)=h(x)$;\\
\\
iii)\quad $\displaystyle\lim_{x\rightarrow 0}\varphi(x)$ exists and is finite.
\end{prop}
\begin{proof}
 This follows immediately from multidimensional Taylor's formula. Note that the second property is true for any polynomial of the first order.
\end{proof}

 Observe that the second property from the Proposition \ref{propositionodd} is a counterpart of \eqref{Vandermonde_identity} with $\lambda=1$. This is the order of homogenity of the corresponding function $h$, therefore, the proof of the Hardy inequality for odd functions will follow the same scheme as the proof of \eqref{Hardyantisymmetric}, with $h$ playing the role of the Vandermonde determinant. We now make this more precise.   

\begin{proof}[Proof of Theorem \ref{tw2}]
 By making use of the Proposition \ref{density}, we let $u$ to be a smooth, compactly supported odd function. Let $h$ be a polynomial from Proposition \ref{propositionodd}. We define the domains $\Omega_+=\{x\in\R^d: h(x)>0\}$ and $\Omega_{-}=\{x\in\R^d: h(x)<0\}=\Omega_+^c$. Clearly $\partial\Omega_+=\partial\Omega_-=\{x\in\R^d: h(x)=0\}$ and $u$ vanishes on $\partial\Omega_{+}$. Moreover, 
 $$\int_{\R^d}|\nabla u(x)|^p\,dx=2\int_{\Omega_+}|\nabla u(x)|^p\,dx$$ 
 and
 $$\int_{\R^d}\frac{|u(x)|^p}{|x|^p}\,dx=2\int_{\Omega_+}\frac{|u(x)|^p}{|x|^p}\,dx,$$ therefore it suffices to prove the corresponding Hardy inequality on the set $\Omega_+$.

 For $\al\in\R,\,\be\geq 0$ we now define the vector fields on $\Omega_+$ by
 $$
T_{\varepsilon}(x)=\frac{\al x}{|x|^p+\varepsilon}-\frac{\be\,\nabla h(x)}{\left(h(x)+\va\right)\left(|x|^{p-2}+\varepsilon\right)}
$$
and
$$
T(x)=\lim_{\varepsilon\rightarrow 0}T_{\varepsilon}(x)=\frac{\al x}{|x|^p}-\frac{\be\,\nabla h(x)}{h(x)|x|^{p-2}}.
$$
The rest of the proof is the same as in the proof of Theorem \ref{tw1}; one only needs to replace $\Vd$ by $h$ and $\lambda$ by one in all the computations therein.
\end{proof}
\begin{proof}[Proof of Theorem \ref{Hardyweighted}]
To deal with weights, we first need to formulate a weighted analogue of the general Hardy-type inequality \eqref{generalHardyT}. Let $D\subset\R^d$ be an opet set, let $T\colon D\to\R^d$ be a $C^1$ vector field and suppose that the $|u(x)|^p T^k(x)|x|^{-\gamma}$ vanish for $x\rightarrow\partial D$ and $x\rightarrow\infty$. Then, integration by parts and Hölder and Young inequalities yield
\begin{align*}
\int_{D}\frac{|u(x)|^p}{|x|^{\gamma}}\divv T(x)\,dx&=-p\int_{D}\frac{|u(x)|^{p-2}u(x)}{|x|^{\gamma}}\langle \nabla u(x),T(x)\rangle\,dx+\gamma\int_{D}\frac{|u(x)|^p}{|x|^{\gamma}}\langle x, T(x)\rangle\,dx\\
&\leq p\int_{D}\frac{|u(x)|^{p-1}|\nabla u(x)||T(x)|}{|x|^{\gamma}}\,dx+\gamma\int_{D}\frac{|u(x)|^p}{|x|^{\gamma+2}}\langle x, T(x)\rangle\,dx\\
&\leq p\left(\int_{D}\frac{|u(x)|^p|T(x)|^{\frac{p}{p-1}}}{|x|^{\gamma}}\,dx\right)^{\frac{p-1}{p}}\left(\int_{D}\frac{|\nabla u(x)|^p}{|x|^{\gamma}}\,dx\right)^{\frac{1}{p}}\\
&+\gamma\int_{D}\frac{|u(x)|^p}{|x|^{\gamma+2}}\langle x, T(x)\rangle\,dx\\
&\leq (p-1)\int_{D}\frac{|u(x)|^p|T(x)|^{\frac{p}{p-1}}}{|x|^{\gamma}}\,dx+\int_{D}\frac{|\nabla u(x)|^p}{|x|^{\gamma}}\,dx\\
&+\gamma\int_{D}\frac{|u(x)|^p}{|x|^{\gamma+2}}\langle x, T(x)\rangle\,dx,
\end{align*}
which leads us to
\begin{equation}\label{generalHardyTweighted}
\int_{D}\frac{|\nabla u(x)|^p}{|x|^{\gamma}}\,dx\geq\int_{D}\frac{|u(x)|^p}{|x|^{\gamma}}\left[\divv T(x)-(p-1)|T(x)|^{\frac{p}{p-1}}-\gamma\frac{\langle x, T(x)\rangle}{|x|^2}\right]\,dx.    
\end{equation}
If necessary, one can also replace $|x|^{\gamma}$ by $|x|^{\gamma}+\varepsilon$ in all the computations above and then let $\varepsilon\rightarrow 0$.

To prove the weighted Hardy inequality for antisymmetric functions, we modify the proof of Theorem \ref{tw1}; we denote $\la=\tfrac{d(d-1)}{2}$, $D=\Omega$ and $T(x)=\frac{\al x}{|x|^p}-\frac{\be\,\nabla\Vd(x)}{\Vd(x)|x|^{p-2}}$. We clearly have $\langle x,T(x)\rangle=\frac{\al-\be\la}{|x|^{p-2}}$, hence, recalling that $t=\left(\tfrac{|x||\nabla\Vd(x)|}{\Vd(x)}\right)^2\geq\la^2$, the function in the square brackets under the integral in the right-hand side of \eqref{generalHardyTweighted} can be transform into the form
$$
f(t;\al,\be,\gamma)=\al(d-p-\gamma)+\be(p-2+\gamma)\la+\be t-(p-1)\left(\al^2-2\al\be\la+\be^2 t\right)^{\frac{p}{2(p-1)}}.
$$
The optimization procedure like in the proof of Theorem \ref{tw1} leads to the constant $C_H(d,p,\gamma)$. The proof for odd functions is similar and will be omitted.
\end{proof}
\section{Rellich inequalities}\label{Section4}
\begin{proof}[Proof of Theorem \ref{WeightedRellichanti}]
In the proof of the Rellich inequality we will follow \cite[Proof of Theorem 7.2.1]{MR3966452}. We assume that $u\in C^{\infty}_c(\R^d)$ is antisymmetric. Let $\Omega$ be the domain deined in \eqref{Omega}. By antisymmetry, $u$ vanishes on the boundary of $\Omega$, but not necessarily near it. However, we can always find a sequence $u_n$ of smooth, compactly supported functions in $\Omega$ such that all derivatives of $u_n$ converge to derivatives of $u$ uniformly in $\Omega$, as $n\rightarrow\infty$. Next,  for $\va>0$ we define
\begin{align*}
u^{(n)}_{\va}(x)&=\left(u_n^2(x)+\va^2\right)^{\frac{p}{2}}-\va^p,\\
w^{(n)}_{\va}(x)&=\left(u_n^2(x)+\va^2\right)^{\frac{p}{4}}-\va^{\frac{p}{2}}.
\end{align*}
The functions $u^{(n)}_{\va}$ and $w^{(n)}_{\va}$ are smooth, compactly supported in $\Omega$ and nonnegative. We calculate
\begin{align*}
\Delta u^{(n)}_{\va}(x)&=p\left(u_n^2(x)+\va^2\right)^{\frac{p}{2}-1}|\nabla u_n(x)|^2+p(p-2)\left(u_n^2(x)+\va^2\right)^{\frac{p}{2}-2}u_n^2(x)|\nabla u_n(x)|^2\\
&+p\left(u_n^2(x)+\va^2\right)^{\frac{p}{2}-1}u_n(x)\Delta u_n(x)\\
&\geq p(p-1)\left(u_n^2(x)+\va^2\right)^{\frac{p}{2}-2}u_n^2(x)|\nabla u_n(x)|^2+p\left(u_n^2(x)+\va^2\right)^{\frac{p}{2}-1}u_n(x)\Delta u_n(x)\\
&=\frac{4(p-1)}{p}|\nabla w^{(n)}_{\va}(x)|^2+p\left(u_n^2(x)+\va^2\right)^{\frac{p}{2}-1}u_n(x)\Delta u_n(x),
\end{align*}
therefore, the above leads to the inequality
$$
-p\int_{\Omega}\frac{\left(u_{n}^2(x)+\va^2\right)^{\frac{p}{2}-1}u_n(x)\Delta u_n(x)}{|x|^{\ga+2(p-1)}}\,dx\geq\frac{4(p-1)}{p}\int_{\Omega}\frac{|\nabla w^{(n)}_{\va}(x)|^2}{|x|^{\ga+2(p-1)}}\,dx-\int_{\Omega}\frac{\Delta u^{(n)}_{\va}(x)}{|x|^{\ga+2(p-1)}}\,dx,
$$
Recall that both $u^{(n)}_{\va}$ and $w^{(n)}_{\va}$ vanish on the boundary of $\Omega$. Hence, remembering that $\la=\tfrac{d(d-1)}{2}$, by the weighted Hardy inequality from Theorem \ref{Hardyweighted} we get
$$
\int_{\Omega}\frac{|\nabla w^{(n)}_{\va}(x)|^2}{|x|^{\ga+2(p-1)}}\,dx\geq\left[(\ga+2p-2)\la+\left(\frac{d^2-\ga-2p}{2}\right)^2\right]\int_{\Omega}\frac{|w^{(n)}_{\va}(x)|^2}{|x|^{\ga+2p}}\,dx
$$
and, integrating by parts twice,
$$
\int_{\Omega}\frac{\Delta u^{(n)}_{\va}(x)}{|x|^{\ga+2(p-1)}}\,dx=\left(\ga+2p-2\right)\left(\ga+2p-d\right)\int_{\Omega}\frac{u^{(n)}_{\va}(x)}{|x|^{\ga+2p}}\,dx.
$$
Letting $n\rightarrow\infty$ and $\va\rightarrow 0$, by Hölder inequality, we obtain
\begin{align*}
-p\int_{\Omega}\frac{|u(x)|^{p-2}u(x)\Delta u(x)}{|x|^{\ga+2(p-1)}}\,dx\leq p\left(\int_{\Omega}\frac{|u(x)|^p}{|x|^{\ga+2p}}\,dx\right)^{\frac{p-1}{p}}\left(\int_{\Omega}\frac{|\Delta u(x)|^p}{|x|^{\ga}}\,dx\right)^{\frac{1}{p}}.
\end{align*}
Summing all the computations above  gives the desired result for the Rellich inequality like \eqref{weightedRellichantisymmetric}, but on $\Omega$. Now, a symmetry argument similar to this from the proof of the Hardy inequality yields the inequality \eqref{weightedRellichantisymmetric} over the whole $\R^d$.

We are able to proof the optimality of the constant for $p=2$. For simplicity, consider the unweighted case. The weighted case can be treated similarly, but the computations are more involved. For $\va>0$ we define the functions
$$
u_{\va}(x)=\begin{cases}
\displaystyle\frac{\Vd(x)}{|x|^{\al}}\text{\qquad if\,\,} |x|\leq 1,\\
 \displaystyle\frac{\Vd(x)}{|x|^{\be}}\text{\qquad if\,\,}|x|>1,
 \end{cases}
$$
where $\al=d/2+\la-2-\va=(d^2-4)/2-\va$ and $\be=(d^2-4)/2+\va$. Of course, the functions $u_{\va}$ do not belong to the appropriate Sobolev space, since its gradient (and in consequence the Laplacian) does not exist on the unit sphere. Nevertheless, since $\Delta(fg)=g\Delta f +2\langle\nabla f,\nabla g\rangle +f\Delta g$, using \eqref{Vandermonde_identity} and $\Delta\Vd=0$ we calculate that for $\rho\in\R$ and $|x|\neq 1$,
$$
\Delta\frac{\Vd(x)}{|x|^{\rho}}=\rho(\rho+2-d-2\la)\frac{\Vd(x)}{|x|^{\rho+2}}.
$$
Therefore, 
\begin{align*}
\int_{\R^d}|\Delta u(x)|^2\,dx&=\left(\al(\al+2-2\la-d)\right)^2\int_{|x|<1}\frac{|\Vd(x)|^2}{|x|^{2(\al+2)}}\,dx\\
&+\left(\be(\be+2-2\la-d)\right)^2\int_{|x|>1}\frac{|\Vd(x)|^2}{|x|^{2(\be+2)}}\,dx, 
\end{align*}
hence, substituting the values of $\al$ and $\be$ we see that
$$
\left[\left(\frac{d^2}{2}-\va\right)\left(\frac{d^2-4}{2}-\va\right)\right]^2\leq\displaystyle\frac{\displaystyle\int_{\R^d}|\Delta u_{\va}(x)|^2\,dx}{\displaystyle\int_{\R^d}\frac{|u_{\va}(x)|^2}{|x|^4}\,dx}\leq\left[\left(\frac{d^2}{2}+\va\right)\left(\frac{d^2-4}{2}+\va\right)\right]^2
$$
and we conclude that the constant $d^4(d^2-4)^2/16$ is optimal by letting $\va\rightarrow 0$ and choosing ,,smooth version'' of functions $u_{\va}$, that is functions which are smooth and converge to $u_{\va}$ in appropriate norms. We omit the technical details.
\end{proof}
\begin{proof}[Proof of Theorem \ref{WeightedRellichodd}]
    The proof is completely analogous to the proof of Theorem \ref{WeightedRellichanti}. We leave the details to the Reader. To verify the optimality of the constant $(d^2-4)^2/16$ in the unweighted case $p=2$, we use ,,smooth versio'' of the trial functions of the form
    $$
u_{\va}(x)=\begin{cases}
\displaystyle\frac{H_{d}(x)}{|x|^{\al}}\text{\qquad if\,\,} |x|\leq 1,\\
 \displaystyle\frac{H_{d}(x)}{|x|^{\be}}\text{\qquad if\,\,}|x|>1,
 \end{cases}
$$
where $H_{d}(x)=\displaystyle\sum_{k=1}^{d}x_k$, $\al=d/2-1-\va$ and $\be=d/2-1+\va$.  \\
\\
\\
\noindent
\textbf{Acknowledgement.} The author would like to thank Konstantin Merz for helpful discussion.
\end{proof}

%\bibliographystyle{acm}
%\bibliography{bib_file_mega}
\end{document}